\documentclass{amsart}

\newtheorem{thm}{Theorem}
\newtheorem{lem}{Lemma}
\newtheorem{cor}{Corollary}
\newtheorem{conj}{Conjecture}

\theoremstyle{remark}
\newtheorem{eg}{Example}

\usepackage{tikz}																	
\usepackage{mathdots}																
\usepackage{enumitem}																
\usepackage{mathabx}																
\usepackage{hyperref}																
\usepackage{cancel}																	
\newcommand{\acknowledgements}[1]{\section*{Acknowledgements} #1 }					
\newcommand{\Rep}{\mathbf{Rep}}														
\newcommand{\C}{\mathbb{C}}															
\newcommand{\E}{\mathbb{E}}															
\newcommand{\Q}{\mathbb{Q}}															
\newcommand{\R}{\mathbb{R}}															
\newcommand{\GL}{\mathrm{GL}}														
\newcommand{\curly}{\mathcal}														
\newcommand{\HH}{\mathrm{H}}														
\newcommand{\G}{\mathbf{G}}															
\newcommand{\A}{\mathbb{A}}															
\newcommand{\Z}{\mathbb{Z}}															
\renewcommand{\H}{\mathbb{H}}														
\newcommand{\w}{\mathrm{w}}															
\newcommand{\intersect}{\cap}														
\newcommand{\kp}{\vdash}															
\renewcommand{\hat}{\widehat}														
\newcommand{\Dirsum}{\bigoplus}														
\newcommand{\hmtpc}{\simeq}															
\newcommand{\isom}{\cong}															
\newcommand{\homeo}{\approx}														
\newcommand{\til}{\widetilde}														
\DeclareMathOperator{\Hom}{Hom}														
\DeclareMathOperator{\Ext}{Ext}														
\DeclareMathOperator{\Mut}{Mut}														
\newcommand{\codim}{\mathrm{codim}}													
\newcommand{\Tr}{\mathrm{tr}}														
\newcommand{\compose}{\circ}														
\newcommand{\bfs}[1]{\ensuremath\mathbf{#1}}										
\newcommand{\dprod}{\displaystyle \mathop{\prod}^{\curvearrowright}}				

\title[q.~dilog.~identities for $n$-cycles]{Quantum dilogarithm identities for $n$-cycle quivers}

\author[J.~Allman]{Justin Allman}

\address{Department of Mathematics, U.S.~Naval Academy, Annapolis, MD, USA}
\email{allman@usna.edu}

\keywords{%
quantum dilogarithm, quiver representation, maximal green sequence
}

\begin{document}

\begin{abstract}
We prove quantum dilogarithm identities for $n$-cycle quivers. By the combinatorial approach of Keller, each side of our identity determines a maximal green sequence of quiver mutations. Thus we interpret our identities as factorizations of the refined Donaldson--Thomas invariant for the quiver with potential. Finally, we conjecture an upper bound on the possible lengths of maximal green sequences.
\end{abstract}

\maketitle


\section{Introduction}
\label{s:intro}

Quantum dilogarithm identities are remarkable equalities of non-commutative power series which have several interpretations in mathematics and physics. In combinatorics they encode partition counting identities akin to the Durfee's square recursion (e.g.~\cite{rraway2018}). In physics, they are quantum versions of identities in supersymmetric gauge theories (e.g.~\cite{lfrk1994}). In geometry, they represent wall-crossing formulas for stability conditions and relate Poincar\'e series in cohomological Hall algebras (e.g.~\cite{mkys2011}). Keller initiated the study of the so-called \emph{refined Donaldson--Thomas invariant} (henceforth DT-invariant), which is the common value of these equalities, by a combinatorial algorithm called \emph{maximal green sequences} of quiver mutations \cite{bk2013.fpsac}. Such sequences were known to physicists as ``finite chambers'' for BPS spectra in supersymmetric quantum conformal field theories. Recent progress has been made in classifying maximal green sequences. For example, the ``No Gap Conjecture'' of \cite{tbgdmp2014} has been established for tame hereditary algebras, and hence for quivers which are acyclic orientations of simply-laced extended Dynkin diagrams \cite{shki2019}. Furthermore, results on the lower bounds for lengths of maximal green sequences in certain types have appeared, e.g.~\cite{agtmks2017,agtmks2018}. 

In this note we continue a program of studying quantum dilogarithm identities and the associated DT-invariants via topology, see \cite{rr2013,jarr2018}. Whereas maximal green sequences produce implicit quantum dilogarithm identities (via an algorithm), our identities are explicitly determined. On one hand, we study identities of Betti numbers in different stratifications of the quiver's representation space. On the other, we establish identities by dimension counting arguments in the associated quantum algebra. We seek to consider a first approach to the non-acyclic case. As such, we study quivers with \emph{potential} (which is a formal polynomial of cyclic paths unique up to cyclic permutation), and where the topological arguments are complicated by \emph{rapid decay cohomology}. 

We study the specific case of the $n$-cycle quiver. Let $\Gamma_n$ denote the quiver below where we set $a_i$ to be the arrow whose head is the vertex $i$, and to which we assign the potential $W = -a_1a_2\cdots a_n$ throughout the rest of the paper.
	\begin{equation}
		\label{eqn:Cn.picture}
	\vcenter
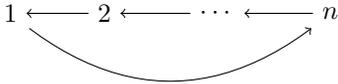
{\hbox{
	\begin{tikzpicture}
	\node (1) at (0,0) {$1$};
	\node (2) at (1.25,0) {$2$};
	\node (dots) at (2.75,0) {$\cdots$};
	\node (n) at (4.25,0) {$n$};
	
	\draw[->] (n) -- (dots);
	\draw[->] (dots) -- (2);
	\draw[->] (2) -- (1);
	\draw[->, domain=0.25:4] plot (\x, {.2*\x*(\x-4.25)});
	\end{tikzpicture}
	}}\nonumber
	\end{equation}

\section{Quiver preliminaries}
\label{s:quiver.preliminaries}
A quiver $Q$ is a directed graph with vertices $Q_0$ and edges $Q_1$ called \emph{arrows}. Every $a\in Q_1$ has a \emph{tail}, $ta$, and \emph{head}, $ha$, in $Q_0$. Given a \emph{dimension vector} $\gamma = (\gamma(i))_{i\in Q_0}$ of non-negative integers, we associate the \emph{representation space}, $\Rep_\gamma(Q)$, and a group, $\G_\gamma$, as follows
\begin{equation}
\label{eqn:defn.repspace.and.G}
{\textstyle \Rep_\gamma(Q) = \Dirsum_{a\in Q_1} \Hom(\C^{\gamma(ta)},\C^{\gamma(ha)}),
	\quad\quad
\G_\gamma = \prod_{i\in Q_0} \GL(\gamma(i),\C).} \nonumber
\end{equation}
The group $\G_\gamma$ acts on $\Rep_\gamma(Q)$ via $(g_i)_{i\in Q_0} \cdot (X_a)_{a\in Q_1} = (g_{ha} X_a g_{ta}^{-1})_{a\in Q_1}$. Let $e_i$ denote the dimension vector with a $1$ at vertex $i\in Q_0$ and zeroes elsewhere. Then for every quiver, we define the $\Z$-bilinear anti-symmetric form $\lambda$ by
\begin{equation}
\label{eqn:lambda.defn}
\lambda(e_i,e_j) = \#\{\text{arrows~} a:i\to j\} - \#\{\text{arrows~} a:j \to i\}
\nonumber
\end{equation}
and extending linearly to all dimension vectors. This is the opposite antisymmetrization of the Euler form. Further, fix an indeterminate $q^{1/2}$. The \emph{quantum algebra} $\A_Q$ is the $\Q(q^{1/2})$-algebra with underlying vector space spanned by the symbols $y_\gamma$, for every dimension vector $\gamma$, subject to the relation that for any two dimension vectors $\gamma_1$ and $\gamma_2$, we have $y_{\gamma_1+\gamma_2} = -q^{-\frac{1}{2}\lambda(\gamma_1,\gamma_2)}y_{\gamma_1}y_{\gamma_2}$, which implies the relation $y_{\gamma_2}y_{\gamma_1} = q^{\lambda(\gamma_1,\gamma_2)}y_{\gamma_1}y_{\gamma_2}$. 
We let $\hat{\A}_Q$ denote the completion of this algebra allowing formal power series in the variables $y_\gamma$.

\subsection{Dynkin quivers}
A \emph{Dynkin quiver} is an orientation of a simply-laced Dynkin diagram, i.e.~of type $A$, $D$, or $E$. In this paper, we need only so-called \emph{equioriented} $A_n$ quivers 
	\begin{equation}
	\vcenter{\hbox{
	\begin{tikzpicture}
	\node (1) at (0,0) {$1$};
	\node (2) at (1.25,0) {$2$};
	\node (dots) at (2.75,0) {$\cdots$};
	\node (n) at (4.25,0) {$n$};
	
	\draw[->] (n) -- (dots);
	\draw[->] (dots) -- (2);
	\draw[->] (2) -- (1);
	\end{tikzpicture}
	}}\nonumber
	\end{equation}
and so henceforth the symbol $A_n$ denotes the quiver above. The simple roots of $A_n$ can be identified with the dimension vectors $e_i$, and we denote by $\Phi(A_n)$ the corresponding set of positive roots. Each $\beta\in\Phi(A_n)$ is realized as a dimension vector for $Q$ and is identified with the interval of positive integers $[k_1,k_2]$ since it can be written uniquely in the form $\beta = \sum_{i=k_1}^{k_2} e_i$ for some $1\leq k_1 \leq k_2 \leq n$.  Hence $|\Phi(A_n)| = \frac{1}{2}n(n+1) = \binom{n+1}{2}$. We let $\beta_0 = [1,n]$ denote the longest root.

For a fixed dimension vector $\gamma$, the work of Gabriel implies that $\Rep_\gamma(Q)$ admits finitely many $\G_\gamma$-orbits if and only if $Q$ is {Dynkin} \cite{pg1972}. Explicitly, Dynkin orbits are in bijection with lists $m=(m_\beta)_{\beta\in\Phi(A_n)}$ of non-negative integers such that $\gamma = \sum_{\beta\in\Phi(A_n)} m_\beta\,\beta$. We call $m$ a \emph{Kostant partition} of $\gamma$ and write $m\vdash\gamma$. Let $\curly{O}_m$ denote the corresponding $\G_\gamma$-orbit. 

\subsection{Counting Betti numbers of orbits} 
Stabilizers of points in $\curly{O}_m$ are conjugate, and hence isomorphic. Let $\G_m$ denote the isomorphism type of such a stabilizer. A standard argument in equivariant cohomology gives an isomorphism $\HH^*_{\G_\gamma}(\curly{O}_m) \isom \HH^*_{\G_m}(pt) = \HH^*(B\G_m)$ where the last equality uses the definition of the so-called \emph{Borel construction}. Here and throughout the paper, all cohomologies have coefficients in $\R$.

The \emph{Poincar\'e series} of a cohomology algebra $H^*$ is $\curly{P}[H^*] = \sum_{k\geq 0} q^{k/2} h_k$ where $h_k=\dim_\R(H^k)$ is the $k$-th Betti number. We set $\curly{P}_d := \curly{P}[\HH^*(B\GL(d,\C))]$. Now $\HH^*(B\GL(d,\C))$ is a polynomial ring in the Chern classes $c_i$ of $\GL(d,\C)$, with $\deg(c_i) = 2i$, whence
\begin{equation}
	\label{eqn:curly.P.d}
	\curly{P}_d = \sum_{k\geq 0} q^k \dim_\R(\HH^{2k}(B\GL(d,\C))) = \prod_{j=1}^d \frac{1}{1-q^j}.
	\nonumber
\end{equation}
Finally, up to homotopy we have that $\G_m \hmtpc \prod_{\beta\in\Phi(A_n)} \GL(m_\beta,\C)$ \cite[Proposition~3.6]{lfrr2002.duke}, and so putting everything together, the K\"unneth isomorphism further implies that 
\begin{equation}
	\label{eqn:poincare.orbit}
	\curly{P}(m) := \curly{P}[\HH^*_{\G_\gamma}(\curly{O}_m)] = \prod_{\beta\in\Phi(A_n)} \curly{P}_{m_\beta}.
\end{equation}

\subsection{$n$-cycle quivers and potentials}
Henceforth assume $n\geq 3$. $\Gamma_n$ admits infinitely many quiver orbits (it is not Dynkin); however, we give two natural methods to stratify $\Rep_\gamma(\Gamma_n)$ into finitely many $\G_\gamma$-stable subvarieties. First, given a Kostant partition $m\kp\gamma$ for the $A_n$ subquiver in $\Gamma_n$, we define the subvariety
	\begin{equation}
		\label{eqn:defn.w.linear.stratum}
		\Sigma^1_m = 
			\left\{ 
			(X_{a_i})_{i=1}^n \in \Rep_\gamma(\Gamma_n) :
				(X_{a_1},\ldots,X_{a_{n-1}}) \in \curly{O}_m \subset \Rep_\gamma(A_n) 
			\right\}.
	\end{equation}
Further, for any integer $1 \leq \ell < n$, write $j = n-\ell$. We rename the vertices and arrows of $\Gamma_n$ according to the picture
	\begin{equation}
		\label{eqn:Cn.w.quad.picture}
		\vcenter{\hbox{
		\begin{tikzpicture}
		\node (1) at (0,1) {$1$};
		\node (topdots) at (2,1) {$\cdots$};
		\node (ell) at (4,1) {$\ell$};
		
		\node (1') at (3.5,0) {$1'$};
		\node (botdots) at (2,0) {$\cdots$};
		\node (j') at (.5,0) {$j'$};
		
		\node at (.95,1.25) {$a_1$};
		\node at (3.25,1.25) {$a_{\ell-1}$};
		\node at (1.45,-.25) {$a_{j'-1}$};
		\node at (2.9,-.25) {$a_{1'}$};
		
		\node at (4.25,.5) {$b$};
		\node at (-.25,.5) {$a$};
	
		\draw[->] (ell) -- (topdots);
		\draw[->] (topdots) -- (1);
		
		\draw[->] (j') -- (botdots);
		\draw[->] (botdots) -- (1');

		\draw[->] (1) -- (j');
		\draw[->] (1') -- (ell);
		\end{tikzpicture}
		}}
	\end{equation}
Notice $A_\ell$ is the subquiver along the top row and $A_{j}$ is the subquiver along the bottom row. We respectively denote the restriction of a dimension vector $\gamma$ to these subquivers by $\gamma|A_\ell$ and $\gamma|A_{j}$. Given Kostant partitions $m \kp \gamma|A_\ell$ and $m'\kp\gamma|A_{j}$ we define
	\begin{equation}
		\label{eqn:defn.w.quad.stratum}
		\Sigma^2_{m,m'} = 
			\left\{
			(X_{a_i})_{i=1}^n  \in \Rep_\gamma(\Gamma_n) :
			\begin{array}{ll}
				(X_{a_1},\ldots,X_{a_{\ell-1}}) \in \curly{O}_m \subset \Rep_{\gamma|A_\ell}(A_\ell)
				 \\
				(X_{a_{1'}},\ldots,X_{a_{j'-1}}) \in \curly{O}_{m'} \subset \Rep_{\gamma|A_{j}}(A_j)	
			\end{array}		
			\right\}.
	\end{equation}
Observe that for every $\gamma$ and every $\ell$, we have
\begin{equation}
	\label{eqn:stratifications}
{ \bigsqcup_{m \kp \gamma} \Sigma^1_m = \Rep_\gamma(\Gamma_n) = \bigsqcup_{m\kp \gamma|A_\ell,\,m'\kp \gamma|A_{j}} \Sigma^2_{m,m'}}
\end{equation}
We remark that $\Sigma^1_m$ and $\Sigma^2_{m,m'}$ coincide for $\ell = n$, but we still choose to make a formal distinction. Corresponding to the potential $W$ we define the regular function 
\begin{gather*}
	\label{eqn:Cn.superpotential}
	W_\gamma:\Rep_\gamma(Q) \to \C \\
	(X_{a_1} ,\ldots, X_{a_n} )
		{\longmapsto}
	-\Tr( X_{a_1} \compose \cdots \compose X_{a_n} )
\end{gather*}
which makes sense because $X_{a_1} \compose \cdots \compose X_{a_n} \in \Hom(\C^{\gamma(1)},\C^{\gamma(1)})$ and is well-defined up to cyclic permutation of the arrows. Moreover, we note that $W_\gamma$ is $\G_\gamma$-invariant. 

Finally, we call the strata defined by \eqref{eqn:defn.w.linear.stratum} \emph{$W$-linear strata} and those defined by \eqref{eqn:defn.w.quad.stratum} \emph{$W$-quadratic strata}. 

\section{Ordering roots}
\label{s:ordering.roots}
We now define a total order on the roots of $A_n$ via the \emph{Auslander--Reiten ($AR$) quiver}; we refer the reader to \cite{rs2014} for more details on the general theory. The quiver $AR(Q)$ has the indecomposable quiver representations as its vertices, which by Gabriel's theorem correspond to positive roots in Dynkin type. Below we display $AR(A_n)$
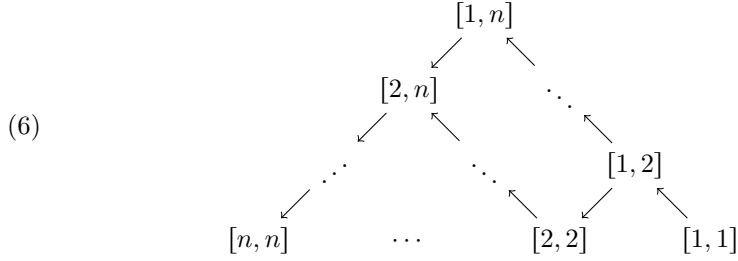
\begin{equation}
	\label{eqn:AR.An}
	\vcenter{\hbox{
	\begin{tikzpicture}
	\node (nn) at (1,-1) {$[n,n]$};
	
	\node (11) at (7,-1) {$[1,1]$};
	\node (22) at (5,-1) {$[2,2]$};
	\node (12) at (6,0) {$[1,2]$};
	
	\node (1n) at (4,2) {$[1,n]$};
	\node (2n) at (3,1) {$[2,n]$};
	
	\node (a) at (2,0) {$\iddots$};
	\node (b) at (3,-1) {$\cdots$};
	\node (d) at (4,0) {$\ddots$};
	\node (e) at (5,1) {$\ddots$};
	
	\draw[->] (11) -- (12);
	\draw[->] (12) -- (22);
	\draw[->] (e) -- (1n);
	\draw[->] (d) -- (2n);
	\draw[->] (1n) -- (2n);
	\draw[->] (a) -- (nn);
	\draw[->] (2n) -- (a);
	\draw[->] (12) -- (e);
	\draw[->] (22) -- (d);
	\end{tikzpicture}
	}}
\end{equation}
%
%
We do not need the arrow information from $AR(Q)$ in this paper, so henceforth we neglect to draw the edges. We observe that for roots $\beta$ and $\beta'$ in the same column of \eqref{eqn:AR.An}, $\lambda(\beta,\beta') = 0$. Hence $y_\beta$ and $y_{\beta'}$ commute in $\hat{\A}_Q$. In general, when $\beta$ appears to the left of $\beta'$, we have $\lambda(\beta,\beta') \geq 0 $.  We seek to impose a total order on $\Phi(A_n)$ which exploits this structure in the quantum algebra. Following \cite{mr2010,rr2013,jarr2018} we impose the rule that
\begin{equation}
	\label{eqn:Reineke.order}
	\beta\prec\beta' \implies \lambda(\beta,\beta') \geq 0.
\end{equation}
Hence one allowed ordering is given by reading from left to right in \eqref{eqn:AR.An}, with any order allowed on roots from the same column. We call this a \emph{Reineke order} on $\Phi(A_n)$.

We remark that the condition $\lambda(\beta,\beta')\geq 0$ is equivalent (in an appropriate sense) to the requirement that $\Hom(M_\beta,M_{\beta'}) = 0 = \Ext(M_{\beta'},M_\beta)$, where $M_\alpha$ denotes the indecomposable quiver representation corresponding to the root $\alpha$, in Reineke's original work \cite[Section~6.2]{mr2010}. This equivalence is proved in the author's joint work with Rim\'anyi \cite[Lemma~5.1]{jarr2018}.

Now, set $\Phi^1 = \Phi(A_n)\setminus\{\beta_0\}$ which we associate with the lefthand side of \eqref{eqn:stratifications}. We give $\Phi^1$ a Reineke order, only forgetting the longest root $\beta_0 = [1,n]$.

For $1\leq \ell < n$, set $\Phi^2_\ell = \Phi(A_\ell) \sqcup \Phi(A_j)$ which we associate with the righthand side of \eqref{eqn:stratifications}. We order $\Phi^2_\ell$ as follows. Roots from $\Phi(A_\ell)$ (and $\Phi(A_j)$) appear in Reineke order. On the other hand, if one of $\beta$ or $\beta'$ is in $\Phi(A_\ell)$ and the other is in $\Phi(A_j)$, then
\begin{equation}
	\label{eqn:neg.order.rule}
	\beta\prec\beta' \implies \lambda(\beta,\beta')\leq 0.
\end{equation}
We observe that an allowed ordering is achieved by aligning the $AR$ quivers of $A_\ell$ and $A_j$ along columns so that they are centered one over the other; i.e.~so that the longest roots $[1,\ell]$ and $[1',j']$ appear in the same column. One checks that again taking columns left to right yields an ordering satisfying \eqref{eqn:Reineke.order} and \eqref{eqn:neg.order.rule}.

\begin{eg}
	\label{eg:A2.A3}
Let $n=5$ and $\ell=3$ (hence $j=2$). Then we have the diagram of $AR$ quivers
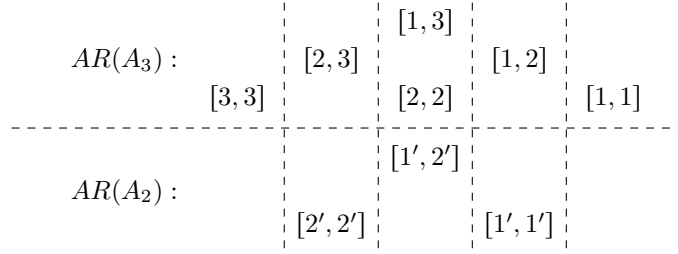
\begin{center}
	\begin{tikzpicture}

	\node (33) at (-.5,.5) {$[3,3]$};
	\node (22) at (2,.5) {$[2,2]$};
	\node (11) at (4.5,.5) {$[1,1]$};
	
	\node (12) at (3.25,1) {$[1,2]$};
	\node (23) at (.75,1) {$[2,3]$};
	\node (13) at (2,1.5) {$[1,3]$};
	
	\node (A3) at (-2,1) {${AR}(A_3):$};
	\node (A2) at (-2,-.75) {${AR}(A_2):$};

	\node (2'2') at (.75,-1.2) {$[2',2']$};
	\node (1'1') at (3.25,-1.2) {$[1',1']$};
	\node (1'2') at (2,-.3) {$[1',2']$};

	\draw[dashed] (2.625,-1.5) -- (2.625,1.8);
	\draw[dashed] (1.375,-1.5) -- (1.375,1.8);
	\draw[dashed] (.125,-1.5) -- (.125,1.8);
	\draw[dashed] (3.875,-1.5) -- (3.875,1.8);
	
	\draw[dashed] (-3.5,0.1) -- (5.5,0.1);
	
	\end{tikzpicture}

\end{center}
Which gives the following allowed total ordering on $\Phi^2_3$
	\[
	[3,3]\prec \underbrace{[2',2'] \prec [2,3]} \prec \underbrace{[1',2'] \prec [2,2] \prec [1,3]} \prec \underbrace{[1',1'] \prec [1,2]} \prec [1,1].
	\]
The braces correspond to sets which commute in $\hat{\A}_{\Gamma_5}$. Within braces, the roots can be permuted to produce another allowed ordering.
\end{eg}

\section{Statement of the main theorem}
\label{s:main.thm}

Given a variable $z$, the \emph{quantum dilogarithm series} $\E(z) \in \Q(q^{1/2})[[z]]$ is
	\begin{equation}
		\label{eqn:E.defn}
		\E(z) 	= 1 + \sum_{d\geq 1} \frac{(-z)^d\, q^{d^2/2}}{\prod_{k=1}^d (1-q^k)} 
				= \sum_{d\geq 0} (-z)^d\,q^{d^2/2}\,\curly{P}_d.
	\end{equation}

\begin{thm}
\label{thm:main}
Let $n\geq 3$, $1\leq \ell <n$, and $j=n-\ell$. In the completed quantum algebra $\hat{\A}_{\Gamma_n}$, we have the following quantum dilogarithm identity
	\begin{equation}
		\label{eqn:main}
		{\textstyle \dprod_{\phi \in \Phi^1} \E(y_\phi) 
			= \dprod_{\psi \in \Phi^2_\ell} \E(y_\psi)}
	\end{equation}
where the arrows indicate the products are taken in the orders specified in Section \ref{s:ordering.roots}.
\end{thm}

We interpret the common value of each side as the refined DT-invariant $\E_{Q,W}$. We remark the identity \eqref{eqn:main} still holds with $n=2$, but the $2$-cycle quiver $\Gamma_2$ is not a so-called \emph{cluster quiver} and $\E_{Q,W}$ is not defined in the sense of \cite{bk2013.fpsac}. In any event, when $n=2$ the identity says that $\E(y_{e_1})\E(y_{e_2}) = \E(y_{e_2})\E(y_{e_1})$ which reflects that the quantum algebra is commutative, a fact which holds for any symmetric quiver.

\section{Quiver mutation and maximal green sequences}
\label{s:mgs}

We provide a brief overview of quiver mutation and green sequences; for more details we refer the reader to \cite{bk2013.fpsac,tbgdmp2014} and references therein. Let $Q$ be a quiver without loops or $2$-cycles, aka a \emph{cluster quiver}. $(Q,F)$ is an \emph{ice quiver} if $F\subset Q_0$ (possibly $F=\emptyset$) is a subset of \emph{frozen vertices} with no arrows between them; we will not be allowed to mutate at frozen vertices. For a non-frozen vertex $i\in Q_0$, define the \emph{mutation of $Q$ at $i$} to be the ice quiver $(\mu_i(Q),F)$ where $\mu_i(Q)$ is obtained from $Q$ by the following steps:
	\begin{enumerate}[label=\arabic*.,leftmargin=*]
	\item For every $k \to i \to j$ (for any other $k,j\in Q_0$), add an arrow $k\to j$,
	\item reverse the orientation of any arrow incident with $i$, and
	\item remove a maximal collection of $2$-cycles \& any arrows between frozen vertices.
	\end{enumerate}
We let $\Mut(Q)$ denote the \emph{mutation class} of $Q$; these are all of the (ice) quivers which can be obtained by sequences of mutations of $Q$.

Given a cluster quiver $Q$ (with no frozen vertices), take a copy of the set of vertices, denote it by $\til{Q}_0$ (for $i\in Q_0$, the corresponding vertex in $\til{Q}_0$ is denoted $\til{i}$), and set $F = \til{Q}_0$. We form the \emph{framed quiver} $\hat{Q}$ with vertices $Q_0\sqcup \til{Q}_0$ and arrows the same as $Q$, except we add an arrow $i \to \til{i}$ for every $i\in Q_0$. 

Given $R\in \Mut(\hat{Q})$, a vertex $i\in R_0$ is \emph{green} if there are no arrows with head $i$ and tail in $\til{Q}_0=F$. The vertex $i$ is \emph{red} if there are no arrows from $i$ to a frozen vertex. A green sequence for $Q$ is a list $\bfs{i} = (i_1,\ldots,i_k)\subset Q_0$, such that $i_1$ is green in $\hat{Q}$ (actually observe that every non-frozen vertex of $\hat{Q}$ is green) and for all $2 \leq j \leq k$, the vertex $i_{j}$ is green in $\mu_{i_{j-1}} \cdots \mu_{i_1}(\hat{Q})$. The sequence $\bfs{i}$ is called a \emph{maximal green sequence} if every non-frozen vertex in $\mu_{\bfs{i}}(\hat{Q}):= \mu_{i_k}\cdots\mu_{i_1}(\hat{Q})$ is red. 

\begin{eg}
	\label{eg:mgs}
Let $Q = \Gamma_4$. We have the following maximal green sequences\footnote{which can be checked in Keller's quiver mutation Java applet \href{https://webusers.imj-prg.fr/~bernhard.keller/quivermutation/}{https://webusers.imj-prg.fr/\textasciitilde bernhard.keller/quivermutation/}} \[ (1,2,1,3,2,1,4,2,1),\quad (1,2,1,3,4,2,1), \quad (1,3,2,4,1,3) \]
of respective lengths $9$, $7$, and $6$. These are respectively the sizes of $\Phi^1$, $\Phi^2_3$, and $\Phi^2_2$, a connection which we now explain.
\end{eg}

Associated to a green sequence $\bfs{i}$, Keller \cite{bk2013.fpsac} defines a product in $\hat{\A}_{Q}$ as follows
	\[
	\E(\bfs{i}) := \E(y_{\delta_k}) \cdots \E(y_{\delta_1})
	\]
where $\delta_j= \sum_{i\in Q_0} b_{j,\tilde{i}} \, e_i$ and $b_{j,\tilde{i}}$ is the number of arrows $j \to \til{i}$ in $\mu_{i_{j-1}} \cdots \mu_{i_1}(\hat{Q})$ ($\delta_1 = e_{i_1}$). We remark that because our definition \eqref{eqn:E.defn} of $\E(z)$ differs from Keller's by the involution $q^{1/2} \mapsto -q^{-1/2}$, our convention \emph{reverses} the order of products in \emph{op.~cit.} When $\bfs{i}$ is maximal green, the \emph{refined DT-invariant} of the quiver (with potential) is $\E_{Q,W} := \E(\bfs{i})$; see \emph{op.~cit.}, \cite{tbgdmp2014}, and references therein. 
When $Q = \Gamma_n$, our topological/geometric proof of Theorem \ref{thm:main} gives several different factorizations of $\E_{Q,W}$ by using the $W$-linear stratification and the $W$-quadratic factorizations (one for each $\ell$). As such, we conjecture several connections to maximal green sequences.

First, recall our method for ordering roots by column aligning the $AR$ quiver(s) in Section \ref{s:ordering.roots}. Now, given an $AR$ diagram for any $A_d$ label each \emph{row} by the vertex $1,\ldots,d$ reading from bottom to top. We remark that in the $W$-quadratic case, this means rows for $A_\ell$ will be labeled $1$ to $\ell$, and rows for $A_j$ will be labeled by $1' = \ell +1$ to $j' = n$.

\begin{conj}
	\label{conj:mgs}
	(a) For any $\ell$, the righthand side of \eqref{eqn:main} is equivalent to $\E(\bfs{i})$ for a maximal green sequence $\bfs{i}$ obtained by recording the row numbers of nodes from right to left and bottom to top in the diagram with $AR(A_\ell)$ and $AR(A_j)$ aligned by columns as in Section \ref{s:ordering.roots}.
	
	(b) The lefthand side of \eqref{eqn:main} is equivalent to $\E(\bfs{i})$ for a maximal green sequence $\bfs{i}$ obtained by recording the row numbers of nodes in $AR(A_n)$ from right to left and bottom to top, except that $[1,n]$ is ignored, and the root $[2,n]$ is considered in row $n$ (not $n-1$).
\end{conj}

\noindent In the statements above, equivalent means ``up to permutation of commuting neighboring factors''.


\begin{eg}
	\label{eg:AR.mgs}
Below are the diagrams corresponding to the maximal green sequences in Example \ref{eg:mgs}. We shorten the notation $[k_1,k_2]$ to $k_1k_2$ and write $4$ instead of $1'$ in the $\ell =3$ case, and $3$ (resp.~$4$) instead of $1'$ (resp.~$2'$) in the $\ell =2$ case.

\begin{center}
	\begin{tikzpicture}
	\node (44) at (0,1) {$44$};
	\node (33) at (1,1) {$33$};
	\node (22) at (2,1) {$22$};
	\node (11) at (3,1) {$11$};
	\node (34) at (.5,1.5) {$34$};
	\node (23) at (1.5,1.5) {$23$};
	\node (12) at (2.5,1.5) {$12$};
	\node (24) at (1,2.5) {$24$};
	\node (14) at (1.5,2.5) {$\cancel{14}$};
	\node (13) at (2,2) {$13$};
	
	\node (r1) at (4,1) {$1$};
	\node (r2) at (4,1.5) {$2$};
	\node (r3) at (4,2) {$3$};
	\node (r4) at (4,2.5) {$4$};
	
	\node (title) at (1.5,0) {$\Phi^1$};
	\node (row) at (4,0) {row};

	\draw[dashed] (-.5,1.25) -- (4.5,1.25);	
	\draw[dashed] (-.5,1.75) -- (4.5,1.75);	
	\draw[dashed] (-.5,2.25) -- (4.5,2.25);
	\draw[dashed] (.25,.5) -- (.25,2.75);	
	\draw[dashed] (.75,.5) -- (.75,2.75);	
	\draw[dashed] (1.25,.5) -- (1.25,2.75);	
	\draw[dashed] (1.75,.5) -- (1.75,2.75);	
	\draw[dashed] (2.25,.5) -- (2.25,2.75);	
	\draw[dashed] (2.75,.5) -- (2.75,2.75);	
	
	\draw (-.5,.5) -- (4.5,.5);
	\draw (3.5,-.5) -- (3.5,2.75);
	
	\draw[thick,->,red] (1 , 1.75) -- (1 , 2.25);
	\end{tikzpicture}
\end{center}

\begin{center}
\begin{minipage}{0.4\textwidth}
\begin{center}
	\begin{tikzpicture}
	\node (33) at (0,1) {$33$};
	\node (22) at (1,1) {$22$};
	\node (11) at (2,1) {$11$};
	\node (23) at (0.5,1.5) {$23$};
	\node (12) at (1.5,1.5) {$12$};
	\node (13) at (1,2) {$13$};
	\node (44) at (1,2.5) {$44$};
	
	\node (r1) at (3,1) {$1$};
	\node (r2) at (3,1.5) {$2$};
	\node (r3) at (3,2) {$3$};
	\node (r4) at (3,2.5) {$4$};
	
	\node (title) at (1,0) {$\Phi^2_3$};
	\node (row) at (3,0) {row};

	\draw[dashed] (-.75,1.25) -- (3.75,1.25);	
	\draw[dashed] (-.75,1.75) -- (3.75,1.75);	
	\draw[dashed] (-.75,2.25) -- (3.75,2.25);
	\draw[dashed] (.25,.5) -- (.25,2.75);	
	\draw[dashed] (.75,.5) -- (.75,2.75);	
	\draw[dashed] (1.25,.5) -- (1.25,2.75);	
	\draw[dashed] (1.75,.5) -- (1.75,2.75);	
	
	\draw (-.75,.5) -- (3.75,.5);
	\draw (2.5,-.5) -- (2.5,2.75);
	\end{tikzpicture}
\end{center}
\end{minipage}
\begin{minipage}{0.4\textwidth}
\begin{center}
	\begin{tikzpicture}
	\node (22) at (0,1) {$22$};
	\node (11) at (1,1) {$11$};
	\node (12) at (0.5,1.5) {$12$};
	
	\node (2'2') at (0,2) {$44$};
	\node (1'1') at (1,2) {$33$};
	\node (1'2') at (.5,2.5) {$34$};
	
	\node (r1) at (2,1) {$1$};
	\node (r2) at (2,1.5) {$2$};
	\node (r3) at (2,2) {$3$};
	\node (r4) at (2,2.5) {$4$};
	
	\node (title) at (.5,0) {$\Phi^2_2$};
	\node (row) at (2,0) {row};

	\draw[dashed] (-.5,1.25) -- (2.5,1.25);	
	\draw[dashed] (-.5,1.75) -- (2.5,1.75);	
	\draw[dashed] (-.5,2.25) -- (2.5,2.25);
	\draw[dashed] (.25,.5) -- (.25,2.75);	
	\draw[dashed] (.75,.5) -- (.75,2.75);	
	
	\draw (-.75,.5) -- (2.5,.5);
	\draw (1.5,-.5) -- (1.5,2.75);
	\end{tikzpicture}
\end{center}
\end{minipage}
\end{center}

\noindent The case $n=4$ was explored in \cite{jarr2018} by thinking of $\Gamma_4$ as the so-called \emph{square product} $A_2 \Box A_2$. The stratifications defined in that work correspond to the $\Phi^2_2$ case above, but the $\Phi^1$ and $\Phi^2_3$ cases above are new explicit factorizations of $\E_{Q,W}$.
\end{eg}

\begin{eg}
	\label{eg:A2.A3.mgs}
	Consider the diagram of Example \ref{eg:A2.A3}. Using that $4=1'$ and $5=2'$, this corresponds to the maximal green sequence $(1,4,2,5,1,3,4,2,1)$ for $\Gamma_5$.
\end{eg}

The \emph{No Gap Conjecture} \cite[Conjecture~2.2]{tbgdmp2014} states that the set of lengths of all possible maximal green sequences forms an interval of integers. This has been proven up to tame (acyclic) types \cite{shki2019}, but is still open for acyclic quivers, $\Gamma_n$ for example. There are maximal green sequences of length $8$ for $\Gamma_4$, e.g.~$(1, 2, 1, 3, 2, 4, 2, 1)$. Although this sequence cannot be achieved by Theorem \ref{thm:main} and Conjecture \ref{conj:mgs}, it fills in an interval $\{6,7,8,9\}$, and one can check there are no maximal green sequences for $\Gamma_4$ of length less than $6$. Motivated by our results, we conjecture an upper bound for such an interval.

\begin{conj}
	\label{conj:max.min.mgs}
	For $\Gamma_n$, the maximal length of a maximal green sequence is $\binom{n+1}{2} - 1 = |\Phi^1|$, and this is achieved by the sequence described in Conjecture \ref{conj:mgs}(b).
\end{conj}

The remainder of the paper is dedicated to the proof of Theorem \ref{thm:main}.

\section{Rapid decay cohomology and quivers with potential}
\label{s:rdc}

\emph{Rapid decay cohomology} is defined for a pair $(X;f)$ where $f:X\to\C$ is a regular function on the variety (or manifold) $X$; see \cite{mkys2011,jarr2018} and references therein. Here, we describe rapid decay cohomology for the space $X = \Rep_\gamma(\Gamma_n)$ and $f = W_\gamma$. Because $W_\gamma$ is a $\G_\gamma$-invariant function, an \emph{equivariant rapid decay cohomology} $\HH^*_{\G_\gamma}(\Rep_\gamma(\Gamma_n);W_\gamma)$ can be defined as follows. Fix $t \in \R$ and let $S_t = \{X \in \Rep_\gamma(\Gamma_n) : \mathrm{Re}(W_\gamma) < t\}$. Then define
	\[
	\HH^*_{\G_\gamma}(\Rep_\gamma(\Gamma_n);W_\gamma) = 
		\lim_{t \to - \infty} 
			\HH^*_{\G_\gamma}\left(\Rep_\gamma(\Gamma_n),W_\gamma^{-1}(S_t)\right).
	\]
The relative cohomologies on the righthand side stabilize at finite $t$ \cite{mkys2011}. Furthermore, by restricting $W_\gamma$, we can similarly define, for any choices of $m$ and/or $m'$, the algebras
	\begin{equation}
		\label{eqn:rdc.strata.defn}
	\HH^*_{\G_\gamma}\left(\Sigma^1_m;W_\gamma | \Sigma^1_m\right)
		\quad\text{~and~}\quad
	\HH^*_{\G_\gamma}\left(\Sigma^2_{m,m'};W_\gamma | \Sigma^2_{m,m'}\right).
	\end{equation}
This section is dedicated to computing the algebras \eqref{eqn:rdc.strata.defn} (hence their Poincar\'e series). To reduce notation, we neglect to write when $W_\gamma$ is restricted to a subspace, which will be clear from context. We need the following technical lemma \cite[Lemma~8.6]{jarr2018}.

\begin{lem}
	\label{lem:reduce.normal.form}
	Let $G$ act on the space $X$. Suppose that $Y \subset X$ is a subspace with isotropy subgroup $G_Y$ and $Z\subset X$ is a $G$-invariant subspace. Further assume that (1) every $G$-orbit in $X$ intersects $Y$, and (2) if $g\in G$ is such that there exists $x\in Y$ with $g\cdot x\in Y$, then $g\in G_Y$. In this scenario, we have $\HH^*_G(X,Z) \isom \HH^*_{G_Y}(Y,Y\intersect Z)$. We call $Y$ a ``normal locus for $X$''. \qed
\end{lem}

\subsection{$W$-linear case}
	\label{ss:rdc.wlin}
We give details for this case, but only sketch the methods in the $\Sigma^2_{m,m'}$ case, which follows the technique of \cite{jarr2018}. Fix a Kostant partition $m\kp\gamma$ for $A_n$ and a point $X_m \in \curly{O}_m\subset\Rep_\gamma(A_n)$. The value of $m_{\beta_0}$ is the rank of the linear map $(X_m)_{a_1} \compose \cdots \compose (X_m)_{a_{n-1}} : \C^{\gamma(n)} \to \C^{\gamma(1)}$, and we assume that the bases are chosen in each $\C^{\gamma(i)}$ so the matrix of this composition has the block form
	\[
			\left(
				\begin{array}{c|c}
						I_{m_{\beta_0}} &  0  \\
						\hline
						0 & 0
				\end{array}
			\right).
	\]
Define the \emph{normal locus} $\nu^1_m \subset \Rep_\gamma(\Gamma_n)$ to be $\left\{ (X_{a_i})_{i=1}^n : (X_{a_i})_{i=1}^{n-1} = X_m\right\}$. We have immediately from the definition that $\nu^1_m \subset \Sigma^1_m$. The hypotheses of Lemma \ref{lem:reduce.normal.form} apply to $G=\G_\gamma$, $X=\Sigma^1_m$, $Y=\nu^1_m$ (with $G_Y \isom \G_m$ as in Section \ref{s:quiver.preliminaries}), $Z = W_\gamma^{-1}(S_t)\intersect \Sigma^1_m := Z_m$. Therefore 
\[
\HH^*_{\G_\gamma}\left(\Sigma^1_m,Z_m \right) \isom \HH^*_{\G_m}\left(\nu^1_m, Z_m \intersect \nu^1_m \right).
\]
Set $w = \gamma(1)\gamma(n)$ and observe that $\nu^1_m \homeo \Hom(\C^{\gamma(1)},\C^{\gamma(n)}) \homeo \R^{2w}$ by identifying $Y\in \nu^1_m$ with its matrix along the arrow $a_n$. In the block form we write
	\[
		Y = \left(
				\begin{array}{c|c}
						Y_0 & A \\
						\hline
						B & C
				\end{array}
			\right)
	\]
where $Y_0$ is an $m_{\beta_0}\times m_{\beta_0}$ complex matrix. From this we have
\[
W_\gamma(Y) = -\Tr((X_m)_{a_1} \compose \cdots \compose (X_m)_{a_{n-1}}\compose Y) = -\Tr(Y_0),
\]
a linear function in the entries of $Y_0$ (hence our naming convention). Taking the real part of $W_\gamma$ is also linear, whence the pair $(\nu^1_m,Z_m\intersect \nu^1_m)$ is homotopy equivalent to one of
\[
\text{I.~}(\R^{2w},\H^{2w}) \text{~when $m_{\beta_0}\neq 0$},
	\quad{\text{or}}\quad
\text{II.~}(\R^{2w},\emptyset) \text{~when $m_{\beta_0} = 0$}.
\]
Fixing $t\ll 0$ so that the rapid decay cohomology stabilizes, we have
	\begin{enumerate}[label={\Roman*}.]
	\item $\HH^*_{\G_\gamma}\left(\Sigma^1_m;W_\gamma \right) \isom \HH^*_{\G_m}(\R^{2w},\H^{2w}) = 0$ since $\R^{2w} \hmtpc \H^{2w}$, or
	\item $\HH^*_{\G_\gamma}\left(\Sigma^1_m;W_\gamma  \right) \isom \HH^*_{\G_m}(\R^{2w},\emptyset) \isom \HH^*_{\G_m}(pt) = \HH^*(B\G_m)$,
	\end{enumerate}
the latter algebra appeared in Section \ref{s:quiver.preliminaries}. 
We remark that the vanishing of $\HH^*_{\G_\gamma}(\Sigma^1_m;W_\gamma)$ when $m_{\beta_0} \neq 0$ is the reason for our definition $\Phi^1 = \Phi(A_n)\setminus\{\beta_0\}$ in Section \ref{s:ordering.roots}.

\subsection{$W$-quadratic case}
	\label{ss:rdc.wquad}
Fix $m\kp \gamma|A_\ell$ and $m'\kp \gamma|A_{j}$. As in the previous subsection, we can use Lemma \ref{lem:reduce.normal.form} to reduce our computation to a normal locus. By appropriately choosing bases, an analogous argument to \cite{jarr2018} shows that restricting $W_\gamma$ to this normal locus is a \emph{quadratic} function (hence our naming convention). Furthermore, the homotopy arguments of \cite[Section~8]{jarr2018} can be adapted to yield isomorphisms for every $d$
\[
\HH^d_{\G_\gamma}(\Sigma^2_{m,m'};W_\gamma ) 
	\stackrel{\isom}{\longrightarrow}
		\HH^{d-2\w(m,m')}(B\G_m \times B\G_{m'})
\]
where $\w(m,m') = m_{[1,\ell]} m'_{[1',j']}$. Hence using \eqref{eqn:poincare.orbit}, a K\"unneth formula implies
\begin{equation}
	\label{eqn:W.quad.Poincare}
	\curly{P}[\HH^*_{\G_\gamma}(\Sigma^2_{m,m'};W_\gamma )]  
		 = q^{\w(m,m')} \curly{P}(m) \curly{P}(m').
	\nonumber
\end{equation}

\section{Kazarian spectral sequence, Poincar\'e series identities}
\label{s:KSS}
We now apply \cite[Theorem~9.2]{jarr2018} to the present context. Let $c(m)$ and $c({m,m'})$ respectively denote the complex codimensions of $\Sigma^1_m$ and $\Sigma^2_{m,m'}$ in $\Rep_\gamma(\Gamma_n)$.

\begin{thm}
	\label{thm:kss}
	There is a spectral sequence $E^{ij}_\bullet$ in rapid decay cohomology which converges to $\HH^{i+j}_{\G_\gamma}(\Rep_\gamma(\Gamma_n);W_\gamma)$, degenerates at the $E_1$ page, and has $E_1$ page (with sums over all $\Sigma^1_m$)
	\begin{equation}
		\label{eqn:kss.E1.Wlinear}
		E^{ij}_1 = 
		\Dirsum_{c(m)=2i} 
				\HH^j_{\G_\gamma}(\Sigma^1_m;W_\gamma)
		= \mathop{\Dirsum_{c(m)=2i,}}_{\text{and~} m_{\beta_0} = 0}
				\HH^j(B\G_m).
	\end{equation}
Moreover, for every $1\leq \ell <n$ there is another spectral sequence with the same properties as above except that the $E_1$ page is determined by taking the direct sum over all $\Sigma^2_{m,m'}$
	\begin{equation}
		\label{eqn:kss.E1.Wquad}
		E^{ij}_1 = 
		\Dirsum_{c(m,m')=2i} 
				\HH^{j}_{\G_\gamma}(\Sigma^2_{m,m'};W_\gamma)
		= \mathop{\Dirsum_{c(m,m')=2i}}
				\HH^{j-2\w(m,m')}(B\G_m \times B\G_{m'}).
	\end{equation}	
\end{thm}

\begin{proof} The proof is analogous to that of \cite[Theorem~9.1]{jarr2018} except that the second equalities in \eqref{eqn:kss.E1.Wlinear} and \eqref{eqn:kss.E1.Wquad} are determined respectively in Sections \ref{ss:rdc.wlin} and \ref{ss:rdc.wquad}.
\end{proof}

\begin{cor}
	\label{cor:P.series.ident}
	Theorem \ref{thm:kss} implies the following identity of Poincar\'e series (for every $1\leq \ell < n$)
	\begin{equation}
		\label{eqn:P.series.ident}
		\mathop{\sum_{m \kp \gamma}}_{m_{\beta_0} = 0} q^{c(m)} \curly{P}(m) 
		= \curly{P}[\HH^*(\Rep_\gamma(\Gamma_n);W_\gamma)] 
		= \mathop{\sum_{m \kp \gamma|A_\ell}}_{m'\kp \gamma|A_{j}} q^{c(m,m')+\w(m,m')} \curly{P}(m)\curly{P}(m')
	\nonumber
	\end{equation}
where the sum on the left is over Kostant partitions for $A_n$ quiver orbits, and the sum on the right is over pairs of Kostant partitions for $A_\ell$ and $A_{j}$ quiver orbits. \qed
\end{cor}

\section{Counting in $\hat{\A}_{\Gamma_n}$ and proof of the main theorem}
\label{s:count.qalg}
Throughout the section, set $y_{e_i} = y_i$ for any vertex $i$ of $\Gamma_n$.

\subsection{$W$-linear strata}
Set $N = |\Phi^1| = \frac{1}{2}n(n+1)-1$, and write $\phi_1 \prec \cdots \prec \phi_N$ as a total order described in Section \ref{s:ordering.roots} for $\Phi^1$. Let $m\kp \gamma$ and let $m_u$ denote the entry of $m$ corresponding to $\phi_u \in \Phi^1$ (by design $m_{\beta_0} = 0$). Moreover, write $\phi_u = \sum_{i=1}^n d^i_u e_i$. Then \cite[Lemma~5.1]{rr2013} implies that in $\hat{\A}_{\Gamma_n}$ we have
	\begin{equation}
		\label{eqn:Wlinear.side.quantum.algebra}
		\begin{gathered}
		y_{\phi_1}^{m_1} \cdots y_{\phi_N}^{m_N} 
				= (-1)^{s_1}\cdot q^{p_1}\cdot q^{-\gamma(1)\gamma(n)} \cdot 
				y_{1}^{\gamma(1)} \cdots y_{n}^{\gamma(n)}, 
					\quad \text{where} \\
		s_1 = {\textstyle \sum_{u=1}^N m_u(\sum_{i=1}^n d^i_u - 1)} 
				\quad \text{and} \quad
		p_1 = {\textstyle c(m) 
			+ \frac{1}{2}\sum_{i=1}^n \gamma(i)^2 
			- \frac{1}{2}\sum_{u=1}^N m_u^2}.
		\end{gathered}	
	\end{equation}
The extra factor of $q^{-\gamma(1)\gamma(n)}$ compared to \cite[Lemma~5.1]{rr2013} results from commuting contributions of $e_n$ in positive roots $[k_1,n]$ past those of $e_1$ in roots $[1,k_2]$ (for $k_1>1$, $k_2<n$) since $[k_1,n] \prec [1,k_2]$ in our order.

\subsection{$W$-quadratic strata}
Set $L = |\Phi(A_\ell)|$, $J=|\Phi(A_j)|$, and $N' = L+J = |\Phi^2_\ell|$. Write a total order $\psi_1\prec \cdots \prec\psi_{N'}$ for $\Phi^2_\ell$ satisfying the conditions in Section \ref{s:ordering.roots}. Let $m_{\psi_i}$ denote the entry, either in $m \kp \gamma|A_\ell$ or in $m'\kp \gamma|A_j$, for the root $\psi_i\in\Phi^2_\ell$. In abuse of notation, we also write $\beta_1 \prec \cdots \prec\beta_L$ as a Reineke order on $\Phi(A_\ell)\subset \Phi^2_\ell$, $\beta'_1\prec \cdots \prec\beta'_J$ as a Reineke order on $\Phi(A_j) \subset \Phi^2_\ell$, and denote the entries from $m$ and $m'$ corresponding to $\beta_u$ and $\beta'_v$ respectively by $m_u$ and $m'_v$. Then define $p'$ by
\begin{equation}
	\label{eqn:Wquad.separate}
	y_{\psi_1}^{m_{\psi_1}} \cdots y_{\psi_{N'}}^{m_{\psi_{N'}}}
		= q^{p'}	\left( y_{\beta_1}^{m_1} \cdots y_{\beta_L}^{m_L} \right) 
					\left( y_{\beta'_1}^{m'_1} \cdots y_{\beta'_J}^{m'_J} \right)
\end{equation}
and calculation reveals that $p' = -\gamma(1)\gamma(n) + \w(m,m')$. Again using \cite[Lemma~5.1]{rr2013}, it follows that \eqref{eqn:Wquad.separate} is further equal to
\begin{multline}
	\label{eqn:Wquad.side.quantum.algebra}
	= (-1)^{s_2}\, q^{p_2}\, q^{p'} 
		\left( y_1^{\gamma(1)} \cdots y_\ell^{\gamma(\ell)} \right) 
		\left( y_{1'}^{\gamma(1')} \cdots y_{j'}^{\gamma(j')} \right) \\
	= (-1)^{s_2}\, q^{p_2}\, q^{p'}\,
		y_1^{\gamma(1)} \cdots y_n^{\gamma(n)}
\end{multline}
with $s_2 = \sum_{u=1}^{N'} m_{\psi_u}(d^i_{\psi_u}-1)$ and $p_2 = c(m,m') + \frac{1}{2} \sum_{i=1}^n \gamma(i)^2 - \frac{1}{2}\sum_{u=1}^{N'} m_{\psi_u}^2$. In the calculation of $p_2$, we have used that
\begin{align*}
c(m,m') &:= \codim_\C(\Sigma^2_{m,m'},\Rep_\gamma(\Gamma_n)) \\
	& = \codim_\C(\curly{O}_m,\Rep_{\gamma|A_\ell}(A_\ell))
		+ \codim_\C(\curly{O}_{m'},\Rep_{\gamma|A_j}(A_j)).
\end{align*}
The above equality follows from the fact that, in the definition \eqref{eqn:defn.w.quad.stratum} of $\Sigma^2_{m,m'}$, no requirements are placed on the maps along the arrows $a$ and $b$ in the picture \eqref{eqn:Cn.w.quad.picture}.

\begin{proof}[Proof of Theorem \ref{thm:main}]
We compute the coefficients of $y_1^{\gamma(1)}\cdots y_n^{\gamma(n)}$ on each side of \eqref{eqn:main}. On the lefthand side, we need to consider
	\begin{multline}
		\label{eqn:lhs1}
		\mathop{\sum_{m \kp \gamma}}_{m_{[1,n]} = 0}
			(-1)^{\sum_{u=1}^N m_u} \, q^{\frac{1}{2}\sum_{u=1}^N m_u^2} \,
			y_{\phi_1}^{m_1} \cdots y_{\phi_N}^{m_N} \, \curly{P}(m) \\
		= (-1)^{\sum_{i=1}^n \gamma(i)} \, 
			q^{\frac{1}{2}\sum_{i=1}^n \gamma(i)^2 - \gamma(1)\gamma(n)} \,
			\mathop{\sum_{m \kp \gamma}}_{m_{[1,n]} = 0}
				q^{c(m)} \curly{P}(m) \,
			y_1^{\gamma(1)} \cdots y_n^{\gamma(n)}		
	\end{multline}
where we used \eqref{eqn:Wlinear.side.quantum.algebra} to obtain the second expression. On the righthand side we have
	\begin{multline}
		\label{eqn:rhs1}
		\mathop{\sum_{m\kp\gamma|A_\ell}}_{m'\kp \gamma|A_j}
			(-1)^{\sum_{u=1}^{N'} m_{\psi_u}}\,
			q^{\frac{1}{2}\sum_{u=1}^{N'} m_{\psi_u}^2}\,
			y_{\psi_1}^{m_1} \cdots y_{\psi_{N'}}^{m_{N'}} \, \curly{P}(m) \curly{P}(m') \\
		= (-1)^{\sum_{i=1}^n \gamma(i)} \, 
			q^{\frac{1}{2}\sum_{i=1}^n \gamma(i)^2 - \gamma(1)\gamma(n)} \\
			\times \mathop{\sum_{m\kp\gamma|A_\ell}}_{m'\kp \gamma|A_j}
						q^{c(m,m') + \w(m,m')} \curly{P}(m) \curly{P}(m') \,
								y_1^{\gamma(1)} \cdots y_n^{\gamma(n)}
	\end{multline}
where the second expression follows from \eqref{eqn:Wquad.side.quantum.algebra}. Finally, Corollary \ref{cor:P.series.ident} implies that \eqref{eqn:lhs1} and \eqref{eqn:rhs1} are the same.
\end{proof}

\acknowledgements{We acknowledge partial support from an Office of Naval Research NARC grant.}


\bibliographystyle{plainurl}
\bibliography{jmabib}

\end{document}